\documentclass{article}
\usepackage{amsmath}
\usepackage{amssymb}
\usepackage{amsthm, amscd}
\usepackage{amsmath,amsthm,amsfonts,amssymb}

\usepackage[usenames]{color}

\newtheorem{thm}{Theorem}[section]
\newtheorem*{thm*}{Theorem}

\newtheorem{lemma}[thm]{Lemma}
\newtheorem{corollary}[thm]{Corollary}
\newtheorem{defn}[thm]{Definition}

\theoremstyle{remark}
\newtheorem{rem}[thm]{Remark}
\newtheorem*{nota}{\bf Notation}

\newcommand{\mbb}{\mathbb}
\newcommand{\fra}{\mathfrak}

\newcommand{\Z}{\mbb{Z}}
\newcommand{\Q}{\mbb{Q}}
\newcommand{\D}{\mathcal{D}}
\newcommand{\I}{\mathcal{I}}

\newcommand{\OR}{\mathcal{O}}
\newcommand{\OM}{\mathcal{O}^{\times}}
\newcommand{\RM}{R^{\times}}

\newcommand{\UTO}{\textrm{UT}_n(\OR)}

\newcommand{\UTR}{\textrm{UT}_n(R)}
\newcommand{\TO}{\textrm{T}_n(\OR)}

\newcommand{\TR}{\textrm{T}_n(R)}
\newcommand{\TRf}{\textrm{T}_n(R,\bar{f})}

\newcommand{\SLZ}{\textrm{SL}_n(\Z)}
\newcommand{\SLR}{\textrm{SL}_n(R)}

\newcommand{\SLO}{\textrm{SL}_n(\mathcal{O})}
\newcommand{\GLO}{\textrm{GL}_n(\mathcal{O})}
\newcommand{\GLR}{\textrm{GL}_n(R)}
\newcommand{\diag}{\textrm{diag}}

\newcommand{\fraA}{\fra{A}}
\newcommand{\fraB}{\fra{B}}
\newcommand{\Lrings}{\mathcal{L}_{\text{rings}}}
\newcommand{\Lgroups}{\mathcal{L}_{\text{groups}}}
\newcommand{\one}{\textbf{1}}
\newcommand{\define}{\stackrel{\text{def}}{=}}
\newcommand{\ds}{\displaystyle}

\usepackage{fancyhdr}

\begin{document}
\title{Bi-interpretability with $\Z$ and models of the complete elementary theories of $\SLO$, $\TO$ and $\GLO$, $n\geq 3$}

\author{Alexei G. Myasnikov, Mahmood Sohrabi\footnote{Address: Stevens Institute of Technology, Department of Mathematical Sciences, Hoboken, NJ 07087, USA. Email: msohrab1@stevens.edu }}

\maketitle
\begin{abstract}\noindent Let $\OR$ be the ring of integers of a number field, and let $n\geq 3$. This paper studies bi-interpretability of the ring of integers $\Z$ with the special linear group $\SLO$, the general linear group $\GLO$ and solvable group of all invertible uppertriangular matrices over $\OR$, $\TO$. For each of these groups we provide a complete characterization of arbitrary models of their complete elementary theories.
  \end{abstract}

\section{Introduction and preliminaries} 
Let $\OR$ be the ring of integers of a number field, and $n\geq 3$ be a natural number. The main problem we tackle in this paper is to characterize arbitrary (pure) groups elementarily equivalent to the special linear group $\SLO$ over $\OR$, the general linear group $\GLO$ over $\OR$ and solvable group of all invertible uppertriangular matrices over $\OR$, $\TO$. 

Since Tarski and Mal'cev there has been many interesting results about elementary equivalence of finitely generated groups and rings. A question dominating research in this area has been if and when elementary equivalence between finitely generated  groups (rings) implies isomorphism. Recently, Lubotzky et. al. \cite{ALM} coined the term {\em first-order rigidity}: a finitely generated group (ring) $A$ is first-order rigid if any other finitely generated group (ring) elementarily equivalent to $A$ is isomorphic to $A$. Indeed a stronger version of rigidity, Quasi Finite Axiomatizability, QFA (See Definition~\ref{QFA:def}), due to A. Nies has been around over the past two decades and has been studied for various classes of groups and rings. Indeed, a large class of nilpotent and polycyclic groups~\cite{Oger2006, Lasserre}, all free metabilian groups of finite rank~\cite{Khelif}, the ring of integers $\Z$ (Sabbagh 2004, \cite{Nies2007}), polynomial rings $\Z[x_1, \ldots, x_m]$~ \cite{Nies2007}, and finitely generated fields of characteristic $ \neq 2$~\cite{Pop} are known to have the property. First-order rigidity of non-uniform higher dimensional lattices in semi-simple Lie groups, e.g. $\SLZ$, has been addressed in~\cite{ALM}, and finitely generated profinite groups are proved to have the property in \cite{Lub}.    

A typical technique in studying whether a finitely generated structure $A$ is QFA or first-order rigid is to study a stronger property: whether $A$ is bi-interpretable with the ring of integers $\Z$ (See Definition~\ref{bi-int:def}, and Theorem~\ref{Nies2007} below). We note that the integral Heisenberg group $\text{UT}_3(\Z)$ is QFA, but it is not bi-interpretable with $\Z$ (\cite{Khelif,Nies2007}). It turns out that studying if a finitely generated  group $G$ is bi-interpretable with $\Z$ or how it fails in that respect is also very useful in studying arbitrary groups (rings) elementarily equivalent to a given one. For example, in many cases arbitrary groups that are elementarily equivalent to $G$ seem to have a very particular structure; They are kind of  ``completions" or ``closures" of $G$ with respect to a ring $R$ elementarily equivalent to $\Z$. When dealing with classical groups or algebras such notions of completion or closure coincide with the classical ones, where completions have the same ``algebraic scheme", but the points are over the ring $R$ as above. In this paper we prove that this is the case for the group $G = \SLZ$, it is bi-interpretable with $\Z$ (Theorem~\ref{sln-biinter:thm} below), so it is first-order rigid. Moreover, any other group $H$ with $G \equiv H$ is isomorphic to $\SLR$ with $R \equiv \mathbb{Z}$.  On the other hand the ``extent" to which a group $G$ fails to be bi-interpretable with $\Z$ also often seems to affect the structure of arbitrary groups elementarily equivalent to $G$. Again it seems such groups are ``deformations" of ``exact completions" or ``exact closures" of $G$ over a ring $R$ as above. It only seems proper that these deformations can usually be captured by cohomological data (See \cite{beleg94,MS2010}). For example, the group $\TO$ is not bi-interpretable with $\Z$, if the ring $\OR$ of integers of a number field has an infinite group of units (though $\TO$ and $\OR$ are mutually interpretable in each other). In this case such failure is modulo the infinite center. We prove in this paper that any group $H$ with $H \equiv \TO$ is an ``abelian deformation" of a group $\TR$ where $R \equiv \OR$ (See Theorem~\ref{TO-elemchar:thm} below). The case of $\GLO$, where $\OR$ has infinite group of units is an exception. Even though, it is not bi-interpretable with $\Z$, in this paper we prove that all its models are of the type $\GLR$.

 As we remarked previously, \cite{ALM} contains a proof of the first-order rigidity of $\SLZ$, $n\geq 3$. The authors also prove that $\SLZ$ is prime. In the same paper authors announce that in a future work they would present a proof of the QFA property for $\SLZ$. The approach in that paper is not via bi-interpretability.

In a sequel to this paper we study the relevant questions, when $\OR$ is replaced by a field $F$, where $F$ is say, a number field (or in general a finitely generated field), or algebraically closed field.

 As to the organization of this paper, all results on $\SLO$ are collected in Section~\ref{slo:sec}, the ones on $\TO$ in Section~\ref{TO:sec} and those on $\GLO$ in Section~\ref{GLO:sec}. We fix our notation and state basic definitions and results in Section~\ref{prelim:sec}.  

\subsection{Preliminaries}
 \label{prelim:sec}

\subsubsection{Basic group-theoretic and ring-theoretic notation} 
For a group $G$ by $Z(G)=\{x\in G:xy=yx, \forall y\in G\}$ we denote the center of $G$. The derived subgroup $G'$ of $G$ is the subgroup of $G$ generated by all commutators $[x,y]=x^{-1}y^{-1}xy$ of elements $x$ and $y$ of $G$. We also occasionally use $x^y$ for $y^{-1}xy$, for $x$ and $y$ in $G$. For any element $g$ of $G$, $C_G(g)=\{ x\in G: [x,g]=\one \}$ is the centralizer of $g$ in $G$.

All rings in this paper are commutative associative with unit. We denote the ring of rational integers by $\Z$ and the field of rationals by $\Q$. By a \emph{number field} we mean a finite extension of $\Q$. By \emph{the ring of integers $\OR$ of a number field $F$} we mean the subring of $F$ consisting of all roots of monic polynomials with integer coefficients. For a ring $R$, by $\RM$ we mean the multiplicative group of invertible (unit) elements of $R$. By $R^+$ we mean the additive group of $R$. 

Consider the general linear group $\GLR$ over a commutative associative unitary ring $R$ and let $e_{ij}$, $1\leq i\neq j\leq n$, be the matrix with $ij$'th entry $1$ and every other entry $0$, and let $t_{ij}=\one+e_{ij}$, where $\one$ is the $n\times n$ identity matrix. Let also $t_{ij}(\alpha)=\one+\alpha e_{ij}$, for $\alpha\in R$.  The $t_{ij}$ as defined above are called \emph{transvections}. Let 
$$T_{ij}\define \{t_{ij}(\alpha):\alpha\in R\}$$ i.e. the $T_{ij}$ are \emph{one parameter subgroups generated by $t_{ij}$ over $R$}. The subgroups $T_{ij}$, where $1\leq i< j \leq n$, generate \emph{the subgroup $\UTR$ of $\GLR$ consisting of all upper unitriangular matrices}. 

For a fixed $n\geq 3$ we order all transvections $t_{ij}$, $1\leq i\neq j\leq n$, in a fixed but arbitrary way and denote the corresponding tuple by $\bar{t}$. If $\beta_1,\ldots,  \beta_m$ lists a set of elements of the ring $R$ we put a fixed but arbitrary order on the set of all transvections $t_{ij}(\beta_k)$, $1\leq i\neq j \leq n$, $k=1,\ldots ,m$, and denote the corresponding tuple by $\bar{t}(\bar{\beta})$.  

Let $\diag[\alpha_1, \ldots ,\alpha_n]$ be the $n\times n$ diagonal matrix with $ii$'th entry $\alpha_i\in \RM$. The subgroup $D_n(R)$ consists precisely of these elements as the $\alpha_i$ range over $\RM$. Now, consider the following diagonal matrices, the \emph{dilations},
\[d_i(\alpha)\stackrel{\text{def}}{=} \diag[1,\ldots ,  \underbrace{\alpha}_{i'\text{th}}, \ldots, 1],\]
and let us set \[d_i\define d_i(-1).\] Clearly the $d_i(\alpha)$ generate $D_n(R)$ as $\alpha$ ranges over $\RM$. 
The elements $d_k(\alpha)$, $1\leq k \leq n$, $\alpha\in \RM$ and $t_{ij}(\beta)$, $1\leq i<j\leq n$, $\beta\in R$, generate \emph{the subgroup $\TR$ of $\GLR$ consisting of all invertible upper triangular matrices}.

\subsubsection{Bi-interpretability, Quasi-finite axiomatizability and Primeness}
For basics of Model Theory our reference is~\cite{hodges}. We denote the first-order language of groups by $\Lgroups$ and the first-order language of unitary rings is denoted by $\Lrings$. When a structure $\fraA$ is definable or interpretable in a structure $\fraB$ with parameters $\bar{b}=(b_1, \ldots, b_n)\in |\fraB|^n$ we say that $\fraA$ is interpretable in $(\fraB, \bar{b})$. 

\begin{defn}\label{bi-int:def}
Consider structures $\fraA$ and $\fraB$ in possibly different signatures. Assume
$\fraA$ is interpretable in $\fraB$ via an interpretation $\Delta$,
$\fraB$ is interpretable in $\fraA$ via an interpretation $\Gamma$, $\tilde{\fraA}$ is the isomorphic copy of $\fraA$ \textit{defined} in itself via $\Gamma \circ \Delta$, and $\tilde{\fraB}$ is an isomorphic copy of $\fraB$ \textit{defined} in itself via $\Delta \circ \Gamma$. We say that 
$\fraA$ is bi-interpretable with $\fraB$ if there exists an isomorphism $\fraA\cong \tilde{\fraA}$ which is first-order definable in $\fraA$ and there exists an isomorphism $\fraB\cong \tilde{\fraB}$ which is first-order definable in $\fraB$. 	
\end{defn}	

%
\begin{defn}
	A structure $\fraA$ is said to be prime if $\fraA$ is isomorphic to an elementary submodel of each model $\fraB$ of $Th(\fraA)$. 
			\end{defn}
	\begin{defn}\label{QFA:def}	Fix a finite signature. An infinite f.g. structure is Quasi Finitely Axiomatizable (QFA) if there exists a first-order sentence $\Phi$ of the signature such that 
		\begin{itemize}
			\item $\fraA\models \Phi$
			\item If $\fraB$ is a f.g. structure in the same signature and $\fraB\models \Phi$ then $\fraA\cong \fraB$.
		\end{itemize} \end{defn}
\begin{thm}[See \cite{Nies2007}, Theorem 7.14]\label{Nies2007} If $\fraA$ is a structure with finite signature which is bi-interpretable (possibly with parameters) with $\Z$, then $\fraA$ is prime. If in addition, $\fraA$ is finitely generated, then it is QFA. \end{thm}

\begin{thm}\label{aut-biint:thm} If $\fraA$ and $\fraB$ are structures bi-interpretable with each other, then $\text{Aut}(\fraA)\cong \text{Aut}(\fraB)$, that is, the automorphism groups of the two structures are isomorphic.\end{thm}

\subsubsection{Bi-interpretability of rings of integers of number fields and $\Z$}
The following is a known result. 
\begin{lemma}\label{O-biint-Z:lem} Assume $\OR$ is the ring of integers of a number field $F$ of degree $m$ and $\beta_1, \ldots , \beta_m$ generate it as a $\Z$-module. Then $(\OR, \bar{\beta})$ and $\Z$ are bi-interpretable. 
 \end{lemma}
\begin{proof}
	By~(\cite{Nies2007}, Proposition~7.12) we need to prove that $\OR$ is interpretable in $\Z$ and there is a definable copy $M$ of $\Z$ in $\OR$ together with an isomorphism $f: \OR \to M$ which is definable in $\OR$.
	
The ring $\OR$ is interpreted in $\Z$ by the $m$-dimensional interpretation $\Delta$: 
$$x=\sum_{i=1}^m a_i \beta_i \mapsto  (a_1,\ldots , a_m)$$ where 
$\Z^m$ is equipped with the ring structure: 
$$ e_i\cdot e_j =_\Z(c_{ij1},c_{ij2},\ldots, c_{ijm})\Leftrightarrow \beta_i \cdot \beta_j =_\OR \sum_{k=1}^m c_{ijk}\beta_k$$
and $e_i= (0,\ldots, 0,\underbrace{1}_{\text{$i$'th}},0,\ldots,0)$, for $i=1,\ldots,m$. On the other hand $\Z$ is defined in $\OR$ without parameters as $\Z \cdot 1_\OR$ by the well-known result of Julia Robinson~\cite{julia}.  So we can take $M= \prod_{i=1}^m \Z\cdot 1_\OR$ with $f(x)$ defined as
$$f(x)=(a_1\cdot 1_\OR, \ldots, a_m\cdot 1_\OR) \Leftrightarrow x=\sum_{i=1}^m a_i \beta_i$$  
which is obviously definable in $\OR$. 

By Theorem~\ref{aut-biint:thm} we can not get ride of the parameters, since $\OR$ is not automorphically rigid while $\Z$ is such.
\end{proof} 
\section{ The case of $\SLO$}\label{slo:sec}
\subsection{Bi-interpretability of $\OR$ and $\SLO$}
  
  We firstly point out that by~\cite{Car-Kell} the group $\SLO$ is boundedly generated by the one parameter subgroups $T_{ij}$ generated by the transvections $t_{ij}$. Also, the transvections satisfy the well-known Steinberg relations:
  \begin{enumerate}
  	\item $t_{ij}(\alpha)t_{ij}(\beta)=t_{ij}(\alpha+\beta), \forall \alpha,\beta\in R.$
  	\item $\ds [t_{ij}(\alpha),t_{kl}(\beta)]=\left\{\begin{array}{ll}
  	t_{il}(\alpha\beta)  & \text{if } j=k \\
  	t_{kj}(-\alpha\beta)  & \text{if } i=l \\
  	\one &\text{if } i\neq l, j\neq k
  	\end{array}\right.$
  	for all $\alpha,\beta\in R$.
  \end{enumerate} 
In addition to these, the $t_{ij}$ satisfy finitely many other relations depending on $\OR$. But, those are not the concern of this paper and we don't need referring to them explicitly.

\begin{lemma}\label{cent-sln:lem}
For $G=\SLO$ or $G=\GLO$ and any transvection $t_{kl}\in G$, $Z(C_G(t_{kl}))=T_{kl}\times Z(G)$ . 
\end{lemma} 
\begin{proof}
For $x=(x_{ij})\in G$ a direct calculation imposing $xt_{kl}=t_{kl}x$ shows that
every non-diagonal entry of the $k$'th column and $l$'th row of $x$ has to be zero, and $x_{kk}=x_{ll}$, that is,
 $$x\in C_G(t_{kl}) \Leftrightarrow \left\{ \begin{array}{ll} x_{kk}=x_{ll} & ~ \\
 x_{ij}=0 &\text{ if } i\neq j, \text{ and either } j=k \text{ or } i=l\end{array}\right.$$ In particular every $t_{il}$, $t_{kj}$, and every $t_{ij}$ where $i\neq k$ and $j\neq l$ belongs to $C_G(t_{kl})$. Therefore, 

\begin{equation*}
\label{centofcent:eqn}
Z(C_G(t_{kl}))\leq \left(\bigcap_{1\leq i\neq l\leq n} C_G(t_{il})\right)\cap \left(\bigcap_{1\leq j\neq k\leq n} C_{G}(t_{kj})\right)\cap \left(\bigcap_{1\leq i\neq k, j\neq l\leq n}C_G(t_{ij})\right)\end{equation*}    
So, we have  
$$x\in Z(C_G(t_{kl})) \Rightarrow \left\{\begin{array}{ll} x_{ij}=0 &\text{ if } i\neq j \text{ and } (i,j)\neq (k,l)\\
x_{ii}=x_{jj} &\text{ if } i\neq l \text{ or } j\neq k\end{array}\right. $$

To finish the proof we need to show that for all $i$ and $j$, $x_{ii}=x_{jj}$. Note that if $x_{kk}\neq x_{ii}$ for some $i\neq k,l$, then without loss of generality we can assume that $x_{ii}=1$ if $i\neq l,k$ and $x_{kk}=x_{ll}=-1$. Then, $x\notin C_{G}(t_{il})$, $i\neq k$ even though $t_{il}\in C_G(t_{kl})$ for such choice of $i$. Therefore, $x_{ii}=x_{jj}$ for all $i$ and $j$. Now, $x\in Z(C_G(t_{kl}))$ if and only if $x_{ii}=x_{jj}$ for all $i$ and $j$, and $x_{ij}=0$, if $i\neq j$, and $(i,j)\neq (l,k)$. Indeed, we proved that
$$Z(C_G(t_{kl}))=T_{kl}\times Z(G).$$
\end{proof} 
\begin{corollary}\label{Tij-defn:cor}	For any $k\neq l$, $T_{kl}$ is definable in $(G,\bar{t})$ if $G=\GLO$ or $G=\SLO$.
\end{corollary}
\begin{proof}
Just note that for $k\neq l$, there exists $j\neq k,l$, such that $$T_{kl}=[Z(C_G(t_{kj})),t_{jl}]=[t_{kj},Z(C_G(t_{jl}))].$$
\end{proof}

\begin{nota}
In the following we view one parameter subgroups $T_{kl}$ as a cyclic two-sorted module: $T_{kl}^\OR\define \langle T_{kl}, \OR, \delta\rangle$, where $\delta$ describes the action $\OR$ on $T_{kl}$, that is,
$$y= \delta ( t_{kl}(\alpha), \beta) \Leftrightarrow y= t_{kl}(\alpha\cdot\beta),  \alpha, \beta \in \OR$$
\end{nota}

\begin{lemma}\label{O-interpret-SLO:lem} For any $\OR$, The two-sorted module $T_{kl}^\OR$ is interpretable in $(G,\bar{t})$, $G=\SLO,\GLO$ where the action of $\OR$ on the $t_{kl}$ respects commutation, i.e. for any $k\neq l$ and $\alpha\in \OR$
	$$[t_{kj}(\alpha),t_{jl}]=[t_{kj},t_{jl}(\alpha)]=t_{kl}(\alpha).$$ 
\end{lemma}
\begin{proof} Fix  $1\leq k<j<l\leq n$. All subgroups $T_{kj}$, $T_{jl}$, $T_{kl}$ are definable, hence is the subgroup $G_{k,j,l}\define \langle T_{kj}, T_{jl}, T_{kl}\rangle$. Note that $G_{k,j,l}\cong \text{UT}_3(\OR)$. So, one can use Mal'cev interpretation of $\OR$ in $\text{UT}_3(\OR)$ to prove the claim for the choice of $k,j,l$. Using Steinberg relations one can easily interpret the actions on all the $T_{ij}$, $i\neq j$ so that the ring action respects commutation. 
	
\end{proof}

\begin{thm}\label{sln-biinter:thm} The ring of integers $\OR$ of a number field of degree $m$ and the group $(\SLO, \bar{t})$ are bi-interpretable. Hence, the ring $\Z$ of rational integers and $(\SLO,\bar{t}(\bar{\beta}))$ are bi-interpretable. \end{thm}	
 
\begin{proof}   Let us recall the standard absolute interpretation of $G=\SLO$ in $\OR$.  Consider the definable subset of $\OR^{n^2}$ consisting of $x=(x_{ij})\in \OR^{n^2}$ defined by $\text{det}(x)=_R1$, where $\text{det}(x)$ is a polynomial in the $x_{ij}$ with integer coefficients and $=_R$ denotes identity in the language of rings. Group product and inversion are also defined by (coordinate) polynomials in the $x_{ij}$. These polynomials do also have integer coefficients. 
 Let us denote this interpretation by $\Gamma$.
 
Let us denote the interpretation of $\OR$ in $G$ obtained in Lemma~\ref{O-interpret-SLO:lem} by $\Gamma$ and for simplicity assume $i=1$ and $k=n$. With this choice, $\OR$ is defined on the definable subset $T_{1n}$ of $(G, \bar{t})$. 

We are about to use bounded generation of $G=\SLO$ by the $T_{ij}$. For what is coming, it is important to have a fixed order on the way we express elements of $G$ as a product of elements from the $T_{ij}$. Indeed, by bounded generation there exists a number $w=w(n,\OR)$ depending  on both $n$ and $\OR$ and a function $f:\{1, \ldots w\} \to \{(i,j): 1\leq i\neq j\leq n\}$, such that, identifying $ij$ with $(i,j)$ 

     $$g\in G \Leftrightarrow \exists \bar{\gamma}\in \OR^w (g= \prod_{k=1}^w t_{f(k)}(\gamma_{f(k)}))$$
 
	Now, consider the copy $\tilde{G}$ of $G$ defined in $G$ via  $ G \xrightarrow{\Delta} \OR \xrightarrow{\Gamma} G $. Indeed, $\tilde{G}$ is a group defined on the (set) Cartesian product $(T_{1n})^{n^2}$, subject to finitely many group theoretic  relations. Given any $g \in G$, it is represented as $g= \prod_{k=1}^{w}g_{f(k)}$, $g_{f(k)}\in T_{f(k)}$. Next, we define a map
$$G \xrightarrow{\phi} \tilde{G}, \quad   g\mapsto \prod_{k=1}^{w}h_{f(k)}$$ where $h_{f(k)}$ is the unique element of $T_{1n}$ such that: 

$$h_{f(k)}\define \left\{\begin{array}{ll} 
g_{f(k)} &\text{if~} i=1 \wedge j=n\\ \\
\left[ t_{1j}, g_{f(k)}\right] &\text{if~} i\neq 1\wedge j=n\\ \\
\left[g_{f(k)},t_{jn}\right]&\text{if~} i=1\wedge j\neq n\\ \\
\left[\left[t_{1i}, g_{f(k)}\right],t_{jn}\right]& \text{if~} i\neq 1\wedge j\neq n \end{array}\right.$$
This is an isomorphism between $G$ and $\tilde{G}$ definable in $(G,\bar{t})$. 

Next, consider:
$$ \OR \xrightarrow{\Gamma} G  \xrightarrow{\Delta} \OR $$
and denote the copy of $\OR$ defined in itself via the composition of interpretations by $\tilde{\OR}$. The copy of $G$ defined in $\OR$ is of the form of the matrix products $g=\prod_{k=1}^{w}t_{f(k)}(\gamma_{f(k)})$, for $\gamma_{f(k)}$ in $\OR$. Since there is a fixed order of representation of elements $g$ of $G$ as products of transvections, for each $1\leq i,j \leq n$, there exists a polynomial $P_{ij}(\bar{y})\in \Z[\bar{y}]$, where $\bar{y}=(y_1, \ldots, y_w)$, such that for every matrix $g=(\beta_{ij})_{n\times n}\in \SLO\subset \OR^{n^2}$
$$ (\beta_{ij})_{n\times n}=g=\prod_{k=1}^{w}t_{f(k)}(\gamma_{f(k)}) \Leftrightarrow \forall 1\leq i,j \leq n ~~ \beta_{ij}=P_{ij}(\bar{\gamma}).$$   
We note that by our choice of interpretation of $\OR$ in $G$, $\tilde{\OR}$ is defined on the subset $T_{1n}$ of $\OR^{n^2}=(\beta_{ij})_{n\times n}$. So, the following sets up an isomorphism between $\OR$ and $\tilde{\OR}$, which is definable in $\OR$  
$$\beta= \beta_{1n} \Leftrightarrow \exists \bar{\gamma} \in \OR^w (\beta =P_{1n}(\bar{\gamma}))$$

Bi-interpretability with $\Z$ with the given parameters follows from above and Lemma~\ref{O-biint-Z:lem}.\end{proof}

\begin{corollary} The group $\SLO$ is QFA and prime for any ring of integers $\OR$.
\end{corollary}

\subsection{Models of the complete elementary theory of $\SLO$}
\begin{thm}\label{sln-elemchar:thm} Assume $H$ is a group and $H\equiv \SLO$ in $\Lgroups$. Then $H\cong \SLR$ for some ring $R\equiv \OR$ as rings.\end{thm}

\begin{proof}
	Assume $\Gamma(\bar{t}(\bar{\beta}))$ is the interpretation of $\Z$ in $(\SLO, \bar{t}(\bar{\beta}))$ introduced in Theorem~\ref{sln-biinter:thm}.  By~(\cite{Nies2007}, Theorem~7.14) there is a formula $\Psi_{st}(\bar{x})$ of $\Lgroups$ that is satisfied by $\bar{t}(\bar{\beta})$ and if $\bar{s}(\bar{\gamma})$ is a sequence of elements of $\SLO$ so that $\SLO\models \Psi_{st}(\bar{s})$, then $\Gamma(\bar{s}(\bar{\gamma}))$, and $\Delta$ as in Theorem~\ref{sln-biinter:thm} bi-interpret $\Z$ and $(\SLO,\bar{s}(\bar{\gamma}))$. Now assume $\{\Psi^i:i\in I\}$ lists all axioms of $\Z$ and for each $i$, $\Psi^i_{\Gamma(\bar{t}(\bar{\alpha}))}$ is the $\Lgroups$ sentence such that $\Z\models \Psi^i \Leftrightarrow \SLO\models \Psi^i_{\Gamma(\bar{t}(\bar{\alpha}))}$. Then, by the above argument for each $i\in I$, $\SLO\models  \forall \bar{x}  (\Psi_{st}(\bar{x}) \to  \Psi^i_{\Gamma(\bar{x})})$.      	
	
	Therefore, there exists a tuple $\bar{u}$ of $H$ such that for each $i\in I$,  $H\models \Psi^i_{\Gamma(\bar{u})}$. This implies that $\Gamma(\bar{u})$ interprets a ring $R$ in $H$ where $R\cong \OR \otimes_\Z \Z^*$, where $\Z^*$ is a model of the theory of $\Z$. By Keissler-Shelah's Theorem there is a non-principal ultrafilter $\D$ on a set $\I$, such that the ultraproducts $\Z^\I/\D \cong (\Z^*)^\I/\D$. Now $\ds \OR^\I/\D \cong \OR \otimes_\Z( \Z^\I/\D)\cong \OR \otimes_\Z ((\Z^*)^\I/\D)\cong R^\I/\D$, proving that $\OR\equiv R$. 
	
	Since $G$ is interpretable in $\OR$, the isomorphism between $\SLO$, as the group of all matrices $n\times n$ matrices over $\OR$ of determinant 1, and itself as the group boundedly generated by one-parameter subgroups is definable in $\OR$. So the same fact holds in $\tilde{G}$. Since the isomorphism $\phi: G \to \tilde{G}$ is definable in $G$, the same first-order fact is expressible in $G$. Now, by the above paragraph $H$ is boundedly generated by some one-parameter subgroups of $H$ generated over the ring $R$. Since all formulas involved in the bi-interpretation of $G$ and $\OR$ are uniform, the fact that $H$ is the group of all $n\times n$ matrices of determinant 1 over the ring $R$ holds in $H$.       
\end{proof}

\section{The case of $\TO$}\label{TO:sec}

\begin{lemma}\label{lem:Tij-in-Tn} Assume $\beta_1, \ldots , \beta_m$ are free generators of $\OR$ as a $\Z$-module. Then for each pair $1\leq k<l \leq n$, the subgroup $T_{kl}$ is definable in $$(G,t_{kl}(\beta_1), \ldots ,t_{kl}(\beta_m), d_k)$$ where $G=\TO$. 
\end{lemma}
\begin{proof} Firstly note 
	\begin{align*}C_G(t_{kl})&=\langle d_p(\alpha), d_k(\alpha)d_l(\alpha), T_{il}, T_{kj}, T_{ij}:\\
	&~~~~~ \alpha \in \OM \wedge (p\neq k,l)\wedge i\neq l \wedge j\neq k \rangle\end{align*} An argument similar to the one in the proof of Lemma~\ref{cent-sln:lem} shows that $Z(C_G (t_{kl}))= T_{kl}\times Z(G)$. Next note that $T_{kl}^2=T_{kl}(2\OR)=[d_k, T_{kl}]= [d_k, Z(C_G(T_{kl}))]$. This shows that $T_{kl}^2$ is a definable subgroup. Hence is $T_{kl}=\langle T^2_{kl} ,t_{kl}(\beta_1), \ldots ,t_{kl}(\beta_m)\rangle$.
\end{proof}

\begin{thm}\label{tnomainbiinter:thm} Let $G=\TO$, assume $\OR$ has a finite group of units $\OM$ with generators $\alpha_1, \ldots \alpha_l$ and let $\beta_1, \ldots , \beta_m$ be free generators of $\OR$ as a $\Z$-module. Then the ring $\OR$ is bi-interpretable with $(G, \bar{t}(\bar{\beta}), \bar{d}(\bar{\alpha}))$.\end{thm} 

\begin{proof}
Recall that the group $G=\TO$ is isomorphic to a semi-direct product $D_n(\OR) \ltimes \UTO$. There is an ordering on the $T_{ij}$ where each element $g$ of $\UTO$ has a unique representation as a product of elements of the $T_{ij}$ in that ordering, i.e. there exists a surjective function 
$$f:\{1,2,\ldots, n(n-1)/2\}\to \{(i,j): 1\leq i<j\leq n \}$$ and unique elements $t_{f(k)}(\gamma_{f(k)}(g))\in T_{f(k)}$ such that $$g= \prod_{k=1}^{n(n-1)/2}t_{f(k)}(\gamma_{f(k)}(g))$$ By assumption all the subgroups $d_k(\OM)$ are definable with parameters since these are all finite and the subgroups $T_{ij}$ are definable with parameters by Lemma~\ref{lem:Tij-in-Tn}. Now the result follows with a similar argument as in the proof of Theorem~\ref{sln-biinter:thm}. 
\end{proof}
\begin{corollary}\label{TO-qfa:cor} Assume $H$ is any finite extension of $\UTO$ in $\TO$ which includes all the $d_i$, $i=1,\ldots,n$. Then $H$ is bi-interpretable with $\OR$ and consequently with $\Z$. \end{corollary}

	\begin{rem} We note that if $\OM$ is infinite then by Corollary 3. of ~\cite{Oger2006}, $\TO$ is not QFA, hence it is not bi-interpretable with $\Z$. However, a group $H$ as in our Corollary~\ref{TO-qfa:cor} is QFA and prime. This also follows from the main theorem of~\cite{Lasserre}.   \end{rem}

\subsection{Models of the complete theory of $\TO$ where $|\OM|<\infty$}

\begin{lemma}\label{utofi:lem} Assume $G =\UTO\ltimes A $ where $A$ is a finite subgroup of $D_n$ containing all the $d_i$. If $H$ is any group such that $H\equiv G$, then $G \cong \UTR \ltimes A$ for some ring $R\equiv \OR$, where the action of $A$ on $\UTR$ is the natural extension of the action of $A$ on $\UTO$.\end{lemma} 
\begin{proof} The commutator subgroup $G'$ is absolutely and uniformly definable in $G$ since it is of finite width. In general $G'$ is a finite index subgroup of $\UTO$ (See proof Lemma~\ref{G':lemma}) and it includes $\text{UT}_n(2\OR)$. Now $\sqrt{G'}\define\{x\in G: x^2\in G'\}$ is uniformly definable in $G$ and this subgroup includes both $\UTO$ and the $d_i$. Hence, by Corollary~\ref{TO-qfa:cor} $\sqrt{G'}$ is bi-interpretable with $\Z$. Hence, $G$ is bi-interpretable with $\Z$. The rest of the proof is similar to that of Theorem~\ref{sln-elemchar:thm}. \end{proof}

As a corollary we get the following statement:
 
\begin{thm}\label{TO-elemchar:thm} If $H$ is any group such that $H\equiv \TO$, where $|\OM|<\infty$ then $H\cong \TR$ for some $R\equiv \OR$. \end{thm}    

\subsection{Models of the complete elementary theory of $\TO$ where $\OM$ is infinite} 
We recall a few well-known concepts and facts from the extension theory and its relationship with the second cohomology group in Section~\ref{cohom-rev:sec}. Readers familiar with this material may proceed to Section~\ref{CoT:sec}. 
\subsubsection{ Extensions and 2-cocycles} \label{cohom-rev:sec}
Assume that $A$ is an abelian group and  $B$ is a group, both written multiplicatively. A function $$f:B\times B\rightarrow A$$ satisfying
\begin{itemize}
	\item $f(xy,z)f(x,y)=f(x,yz)f(y,z)$, \quad $\forall x,y,z\in B,$
	\item $f(\one)=f(x,\one)=\one$, $\forall x\in B$,
\end{itemize}
is called a \textit{2-cocycle}. If $B$ is abelian a 2-cocycle $f:B\times B\rightarrow A$ is \textit{symmetric} if it also satisfies the identity:
$$f(x,y)=f(y,x)\quad \forall x,y \in B.$$
By an extension of $A$ by $B$ we mean a short exact sequence of groups
$$\one\rightarrow A \xrightarrow{\mu} E \xrightarrow{\nu} B \rightarrow \one,$$
where $\mu$ is the inclusion map. The extension is called \emph{abelian} if $E$ is abelian and it is called \emph{central} if $A\leq Z(E)$. A \textit{2-coboundary}\index{ $2$-coboundary} $g:B\times B \rightarrow A$ is a 2-cocycle satisfying :
$$\psi(xy)=g(x,y)\psi(x)\psi(y), \quad \forall x,y\in B,$$
for some function $\psi:B\rightarrow A$. One can make the set $Z^2(B,A)$ of all 2-cocycles  and the set $B^2(B,A)$ of all 2-coboundaries into
abelian groups in an obvious way. Clearly $B^2(B,A)$
is a subgroup of $Z^2(B,A)$. Let us set $$H^2(B,A)= Z^2(B,A)/B^2(B,A).$$
Assume $f$ is a 2-cocycle. Define a group $E(f)$ by $E(f)=B\times A$ as sets with the multiplication
$$(b_1, a_1)(b_2, a_2)=(b_1b_2, a_1a_2f(b_1,b_2)) \quad \forall a_1,a_2\in A,\forall b_1,b_2 \in B.$$
The above operation is a group operation and the resulting extension is central. It is a well known fact that there is a bijection between the equivalence classes of central
extensions of $A$ by $B$ and elements of the group $H^2(B,A)$ given by assigning $f \cdot B^2(B,A)$ the equivalence class of $E(f)$. We write $f_1\equiv f_2$ for $f_1,f_2\in Z^2(B,A)$ if they are \emph{cohomologous}, i.e., if $f_1\cdot B^2(B,A)=f_2\cdot B^2(B,A)$.

If $B$ is abelian $f\in Z^2(B,A)$  is symmetric if and only if it arises from an abelian extension of $A$ by $B$. As it
can be easily imagined there is a one to one correspondence between the equivalent classes of abelian extensions
and the quotient group $$Ext(B,A)=S^2(B,A)/(S^2(B,A)\cap B^2(B,A))\index{ $Ext(B,A)$},$$ where $S^2(B,A)$\index{ $S^2(B,A)$} denotes the group of symmetric
2-cocycles. For further details we refer the reader to ~(\cite{robin}, Chapter 11).
\subsubsection{CoT 2-cocycles}\label{CoT:sec}
Assume $A=T\times F$ is an abelian group where $T$ is torsion and $F$ is torsion-free and let $B$ be an abelian group. We know that $Ext(A,B)\cong Ext(T,B)\oplus Ext(F,B)$, so any symmetric 2-cocycle $f\in S^2(A,B)$ can be written as $f\equiv f_1\cdot f_2$, where $f_1\in S^2(A,B)$, $f_2\in S^2(T,B)$. The symmetric 2-cocycle $f$ is said to be a \emph{coboundary on torsion} or $\emph{CoT}$ if $f_2$ is a 2-coboundary.
\subsubsection{Non-split tori and abelian deformations of $\TR$} \label{abdef:sec}

Consider $\TR$ and the torus $D_n(R)$. The subgroup $D_n(R)$ is a direct product $(\RM)^n$ of $n$ copies of the multiplicative group of units $\RM$ of $R$. The center $Z(G)$ of $G$ consists of diagonal scalar matrices $Z(G)=\{\alpha \cdot \one: \alpha\in \RM\}\cong \RM$, where $\one$ is the identity matrix. It is standard knowledge that $Z(G)$ is a direct factor of $D_n(R)$, i.e. there is a subgroup $B\leq D_n$ such that $D_n=B \times Z(G)$. Now we define a new group just by deforming the multiplication on $D_n$. Let $E_n=E_n(R)$ be an arbitrary abelian extension of $Z(G)\cong \RM$ by $D_n/Z(G)\cong (\RM)^{n-1}$. As it is customary in extension theory we can assume $E_n=D_n=B \times Z(G)$ as sets, while the product on $E_n$ is defined as follows:
$$(x_1,y_1)\cdot (x_2,y_2)=(x_1x_2,y_1y_2f(x_1,x_2)),$$
for a symmetric 2-cocycle $f\in S^2(B,Z(G)).$

\begin{rem}\label{cocycles:rem} Indeed any abelian extension $E_n$ of $\RM$ by $(\RM)^{n-1}$ is uniquely determined by some symmetric 2-cocycles $f_i\in S^2(\RM,\RM)$, $i=1, \ldots ,n-1$ up to equivalence of extensions due to the fact that $ Ext((\RM)^{n-1},\RM)\cong \prod_{i=1}^{n-1} Ext(\RM,\RM)$.  So if the $f_i$ are the defining 2-cocycles for $E_n$ above we also denote the group $E_n$ obtained above by $D_n(R, f_1, \ldots , f_{n-1})$ or $D_n(R,\bar{f})$.\end{rem}

We are now ready to define abelian deformations $\TRf$ of $\TR$ via generators and relations. 

For each $i=1, \ldots, n-1$ pick $f_i\in S^2(\RM,\RM)$. Then, an \emph{abelian deformation} $\textrm{T}_n(R,f_1, \ldots, f_{n-1})$ of $\TR$ is defined as follows.

$\TRf$ is the group generated by 
$$\{d_i(\alpha),t_{kl}(\beta): 1\leq i \leq n, 1\leq k<l\leq n, \alpha\in \RM , \beta\in R\},$$

with the defining relations:
\begin{enumerate}
	\item $t_{ij}(\alpha)t_{ij}(\beta)=t_{ij}(\alpha+\beta).$
	\item \begin{equation*}[t_{ij}(\alpha),t_{kl}(\beta)]=\left\{\begin{array}{ll}
	t_{il}(\alpha\beta)  & \text{if } j=k \\
	t_{kj}(-\alpha\beta)  & \text{if } i=l \\
	\one &\text{if } i\neq l, j\neq k
	\end{array}\right.
	\end{equation*}
	\item If $1\leq i\leq n-1$, then $d_{i}(\alpha)d_{i}(\beta)=d_{i}(\alpha\beta)diag(f_i(\alpha,\beta))$, where\\ $diag(f_i(\alpha,\beta))\define d_1(f_i(\alpha,\beta))\cdots d_n(f_i(\alpha,\beta)), $
	\item $[d_i(\alpha),d_j(\beta)]=\one$
	\item \begin{equation*}
	d_k(\alpha^{-1})t_{ij}(\beta)d_k(\alpha)=\left\{\begin{array}{ll}
	t_{ij}(\beta)  & \text{if } k\neq i, k\neq j \\
	t_{ij}(\alpha^{-1}\beta) &\text{if } k=i\\
	t_{ij}(\alpha\beta)&\text{if } k=j
	\end{array}\right.
	\end{equation*}
\end{enumerate}


\begin{lemma}
	\label{well-def:lem}
	The set $\TRf$ is a group for any choice of $f_i\in S^2(\RM,\RM)$.
\end{lemma}
\begin{proof} 
The $t_{ij}(\beta)$ generate a group $G_u\cong{\UTR}$ by relations (1.) and (2.). The $d_i(\alpha)$ generate an abelian group $E_n$ by (3.) and (4.). Note that both of the above are closed under group operations and $G_u\cap E_n=\one$. By (5.) $G_u$ is stable under the action of $E_n$ by conjugation which is described by (5.) itself, i.e. (5.) describes a homomorphism $\psi_{n,R}:E_n \to Aut(G_u)$ so that $\TRf = E_n\ltimes_{\psi_{n,R}} G_u$, as an internal product, and $ker(\psi_{n,R})=Z(G)=\{diag(\alpha): \alpha\in \RM\}$.\end{proof}
\begin{lemma}\label{G':lemma} Assume $R$ is a commutative associative ring with unit. Then the derived subgroup $G'$ of $G=\TRf$ is the subgroup of $G$ generated by
	$$X=\{t_{i,i+1}((1-\alpha)\beta), t_{kl}(\beta):1\leq i\leq n-1, 1< k+1<l\leq n, \alpha\in\RM, \beta\in R\}.$$\end{lemma} 

\begin{proof}
	Let $N$ denote the subgroup generated by $X$ and $G_u$ the subgroup generated by all the $t_{ij}(\beta)$, $\beta\in R$. Each $t_{kl}(\beta)$, with $l-k\geq 2$ is already a commutator by definition, and $$d_i(\alpha^{-1})t_{i,i+1}(-\beta)d_{i}(\alpha)t_{i,i+1}(\beta)=t_{i,i+1}((1-\alpha^{-1})\beta),$$for any $\alpha\in \RM$ and $\beta\in R$, hence $N\leq G'$. To prove the reverse inclusion firstly note that since $G/G_u$ is abelian, $G'\leq G_u\cong \UTR$. Now, pick $x,y\in G$. Then, $x=x_1x_2$ and $y=y_1y_2$, where $x_1,y_1\in D_n(R,\bar{f})$, and $x_2,y_2\in G_u$. Now, 
	\begin{align*}[x,y]&=[x_1x_2,y_1y_2]\\
	&=[x_1,y_1]^{z_1}[x_1,y_2]^{z_2}[x_2,y_1]^{z_3}[x_2,y_2]^{z_4},\\
	&= [x_1,y_2]^{z_2}[x_2,y_1]^{z_3}[x_2,y_2]^{z_4}\end{align*}
	for some $z_i\in G$, $i=1, \ldots 4$. The commutator $[x_2,y_2]\in (G_u)'$, where $(G_u)'$ is characteristic in $G_u$ so normal in $G$. Therefore, $[x_2,y_2]^{z_4}$ is a product of $t_{ij}(\beta)$, $i+1<j$. The commutators $[x_2,y_1]$ and $[x_1,y_2]$ are of the same type. Let us analyze one of them. Indeed, $x_2=d_1(\alpha_1)\cdots d_n(\alpha_n)$, and $y_1=t_{12}(\beta_{12})\cdots t_{1n}(\beta_{1n})$. So, $[x_2,y_1]$ is a product of conjugates of commutators of type $[d_k(\alpha),t_{ij}(\beta)]$. In case that $j>i+1$ this is conjugate of a $t_{ij}(\beta)\in G_u'$ which was dealt with above and is an element of $N$. It remains to analyze the conjugates of $t=t_{i,i+1}((\alpha-1)(\beta))$. Consider $z=xy$, $x=d_1(\alpha_1)\cdots d_n(\alpha_n)\in E_n$ , $y\in G_u$. Then $t^x= t_{i,i+1}((\alpha-1)\alpha_i^{-1}\alpha_{i+1}\beta) \in N$ and $N$ is normalized by $y$. This completes the proof.       
\end{proof}

\begin{corollary}\label{G'-def:cor} If $G=\TRf$, then both $G'$ and $\sqrt{G'}$ are absolutely definable in $G$.\end{corollary}

\begin{rem}\label{split:rem}Assume for an abelian group $A$ we have $A\cong T \times B$ where $T$ and $B$ are some subgroups of $A$. Consider a symmetric 2-cocycle $f: A \to A$. By abuse of notation we consider $f$ as $f: T \cdot B \to T \times B$. Then $f$ is cohomologous to $(g_1g_2, h_1h_2)$ where $g_1\in S^2(T,T))$, $g_2\in S^2(T,B)$, $h_1\in S^2(B,T)$ and finally $h_2\in S^2(B,B)$. We will use this notation in the following. 
	
\end{rem}

\subsubsection{Characterization Theorem} 
\begin{thm}\label{TnMain:thm}Assume $H$ is a group. If $H\equiv \TO$ as groups, then  $H\cong \TRf$ for some $R\equiv \OR$ and CoT 2-cocycles $f_i$.\end{thm}
\begin{proof}
The subgroup $\sqrt{H'}$ (See proof of Lemma~\ref{utofi:lem}) is definable in $H$ by the same formulas which define $\sqrt{G'}\cong \UTO \rtimes A$ in $G$, where $A$ is a finite subgroup of $D_n$ including all the $d_i$. Therefore, by Lemma~\ref{utofi:lem}, $\sqrt{H'}\cong \UTR \rtimes A$, for some $R\equiv \OR$. 

For each $k=1,\ldots, n$ the subgroup $\Delta_k(G)=d_k(\OM)\cdot Z(G)$ is definable in $G$ as the subgroup of $D_n$ consisting of all $x$:
$$t_{ij}^x=\left\{\begin{array}{ll}
	t_{ij}  & \text{if } i\neq k, j\neq k \\
	t_{ij}(\alpha), \alpha\in \OM &\text{if } i=k \text{ or } j=k
\end{array}\right.$$
Note that the above is expressible by $\Lgroups$-formulas. Therefore, for each $k$ there exists an interpretable isomorphism $\Delta_k/Z(G) \to \OM$. One can also express in $\Lgroups$, via the interpretable isomorphisms mentioned above, that $Z(G)= \{\prod_{k=1}^n d_k(\alpha):\alpha\in \OM\}$.

The facts that $D_n \cap \UTO =\one$ and that $\UTO$ is normal in $\TO$ are also expressible using $\Lgroups$-formulas. Now moving to $H$, the same formulas define a subgroup $E_n=E_n(\bar{e})$ of $H$, so that $E_n(\bar{e})/Z(H)\cong (\RM)^{(n-1)}$ and $Z(H)\cong \RM$, $H_u\cap E_n=\one$ and imply that the action of $E_n$ on $H_u\lhd H$ is an extension of the action of $D_n$ on $\UTO$. This proves that $H\cong \UTR \rtimes E_n$, where $E_n\cong D_n(R, \bar{f})$.

The torsion subgroup $T(\Delta_i(G))$ of $\Delta_i(G)$ is finite. Let $N$ be its exponent. For all $n$, the sentences $\forall x\in \Delta_i(G)( x^n=1 \to x^N=1)$ hold in $G$, hence in $H$. So, the formula $x^N=1$ defines $T(\Delta_i(G))$ in $G$ as well as $T(\Delta_i(H))$ in $H$. Hence, $ T(\Delta_i(H))\cong T(\Delta_i(\OM))\cong T(\OM)\times T(\OM)$. In addition, the following holds in $G$, and consequently in $H$:
$$\forall x\in \Delta_i(G),  \exists y\in T(\Delta_i(G)), \exists z\in Z(G)( x^N\in Z(G)  \to (x=yz ))).$$
This ensures that in $H\equiv G$ each $f_i$ defining the $\Delta_i(H)$ as an extension of $Z(H)\cong \RM$ by $\Delta_i(H)/Z(H)\cong \RM $is CoT.

 \end{proof}
\subsubsection{Sufficiency of the characterization}
In this section we shall prove that the necessary condition proven in Theorem~\ref{TnMain:thm} is also sufficient.
 
We will need to state a few well-known definitions and results.

Let $B$ be an abelian group and $A$ a subgroup of $B$. Then $A$ is called a \emph{pure subgroup of $B$} if $\forall n\in \mbb{N}$, $nA=nB\cap A$.
\begin{lemma}\label{pure:lem}Let $A\leq B$ be abelian groups such that the quotient group $B/A$ is torsion-free. Then $A$ is a pure subgroup of $B$.\end{lemma}
\begin{proof}One direction is trivial. For the other direction assume that $g\in nB\cap A$. Then there is $h\in B$ such that $g=nh$. to get a contradiction assume that $h\notin A$. Then $g=nh\notin A$ since $B/A$ is torsion free. A contradiction! So $h\in A$, therefore $g=nh\in nA$.\end{proof}

An abelian group $A$ is called \emph{pure-injective} if $A$ is a direct summand in any abelian group $B$ that contains $A$ as a pure subgroup.

The following theorem expresses a connection between pure-injective groups and uncountably saturated abelian groups.
\begin{thm}[\cite{eklof}, Theorem 1.11]\label{ekthm} Let $\kappa$ be any uncountable cardinal. Then any $\kappa$-saturated abelian group is pure-injective.\end{thm} 

\begin{rem}
	\label{finducef*:rem}
	Assume $A$ and $B$ are abelian group and $f\in S^2(B,A)$. Let $\D$ be an ultrafilter on a set $I$. Let $A^*$ and $B^*$ denote the ultrapowers of $A$ and $B$, respectively, over $(I,\D)$. Then $f$ induces a natural 2-cocycle $f^*\in S^2(B^*,A^*)$ representing an abelian extension of $A^*$ by $B^*$ (See Lemma 7.1 of~\cite{MS2009} for details). \end{rem}
\begin{lemma}\label{saturated-split:lem}  Assume $f\in S^2(\OM, \OM)$ is CoT and $(I,\D)$ is an ultrafilter so that ultraproduct $(\OM)^*$ of $\OM$ over $\D$ is $\aleph_1$-saturated. Then the 2-cocycle $f^* \in S^2((\OM)^*,(\OM)^*)$ induced by $f$ is a 2-coboundary.\end{lemma}
\begin{proof}
	Firstly, by Dirichlet's units theorem $\OM= T \times B$, where $T$ is finite, and $B$ is f.g. and torsion free, if it is not trivial. Note that $(\OM)^*=(\OR^*)^\times$, hence $T((\OM)^*)=(T(\OM))^* \cong T$, where $T(\RM)$ is the maximal torsion subgroup of the group of units $\RM$ of the ring $R$. Then $\OR^* \cong T \cong B^*$. If $B$ is trivial, there is nothing to prove. Assume it is not trivial. Since $\OM$ is infinite, $(\OM)^*$ is $\aleph_1$-saturated. Since $T^*$ is finite abelian and $B^*$ is torsion-free  $Ext(B^*,T^*)=\one$. The assumption that $f$ is CoT implies that the induced cocycle $f^*\in S^2((\RM)^*,(\RM)^*)$ is CoT. Then, using the notation of Remark~\ref{split:rem}, we have that $g_1^*\in S^2(T^*,T^*)$ and $g^*_2\in S^2(T^*,B^*)$ are both coboundaries. The 2-cocycle  $h_1^*\in S^2(B^*,T^*)$ is a coboundary since $B^*$ is torsion-free abelian and $T^*$ is finite abelian. The 2-cocycle $h_2^*\in S^2(B^*,B^*)$ is also a coboundary since $B^*$ is pure-injective by Theorem~\ref{ekthm} and also $B^*$ is a pure subgroup in an extension represented by $h_2^*$ because of Lemma~\ref{pure:lem}. Consequently, $f^*$ is a coboundary.
\end{proof}

\begin{thm} Any group $\TRf$, where $R\equiv \OR$ and each $f_i$ is CoT is elementarily equivalent to $\TO$.\end{thm}

\begin{proof}  Let $(I,\D)$ be an $\aleph_1$-incomplete ultrafilter. As usual, by $C^*$ we mean the ultrapower $C^I/\D$ of a structure $C$. Then $B((R^*)^\times)=B((\RM)^*)=B^*(\RM)$ is either trivial or $\aleph_1$-saturated. If $f^*_i\in S^2((\RM)^*,(\RM)^*)$ denotes the 2-cocycle induced by $f_i$ then for each $i=1, \ldots,n-1$, $f^*_i$ is a 2-coboundary by Lemma~\ref{saturated-split:lem}. The fact that $T^*_n(R,\bar{f})\cong T_n(R^*,\bar{f^*})$ requires only some routine checking. Therefore, \[T^*_n(R,\bar{f})\cong T_n(R^*, \bar{f^*})\cong T_n(R^*)\cong T_n^*(R).\]  
	This concludes the proof utilizing Keisler-Shelah's Theorem.\end{proof}
\section{ The case of $\GLO$}\label{GLO:sec}
\subsection{Bi-interpretability with $\OR$}
\begin{thm}\label{glnmain:thm} Assume $\OR$ has a finite group of units $\OM=\{\alpha_i: i=1,\ldots, m\}$ and let  $d_1(\bar{\alpha})=(d_1(\alpha_1), \ldots, d_1(\alpha_m))$. Then $(\GLO,\bar{t}, d_1(\bar{\alpha}))$ and $\OR$ are bi-interpretable.
\end{thm}

\begin{proof}
We note $\GLO$ is boundedly generated by all the $T_{ij}$ and the finite subgroup $d_1(\OM)\define \langle d_1(\alpha_i): i=1, \ldots,m\rangle$. All these subgroups are definable with the given parameters. The rest of the proof is similar to that of Theorem~\ref{sln-biinter:thm}.\end{proof}
\begin{corollary} If $\OR$ has a finite group of units, then $\GLO$ is QFA and prime.\end{corollary}

\begin{rem}
	If $\OM$ is infinite by Corollary 3 of \cite{Oger2006} and Theorem~\ref{Nies2007}, $\GLO$ is not bi-interpretable with $\Z$. In the following we provide an alternative proof of this fact utilizing Theorem~\ref{aut-biint:thm}.
\end{rem}
\begin{thm}If $\OR$ has an infinite group of units, then $\GLO$ is not bi-interpretable with $\Z$.\end{thm}

\begin{proof} Indeed we prove that $\GLO$ and $\OR$ are not bi-interpretable for any finite set of parameters picked in $\GLO$. So to derive a contradiction assume that $(\GLO, \bar{c})$ and $\OR$ are bi-interpretable for a tuple $\bar{c}$ of constants in $\GLO$. Since $\OR$ and $\Z$ are bi-interpretable there are constant $\bar{e}$ in $\GLO$ where $(\GLO,\bar{c},\bar{e})$ and $\Z$ are bi-interpretable. Now take a countable non-standard model $S$ of $\Z$ with countably many automorphisms and a free $\Z$-basis $\bar{\beta}$ of $\OR$. Consider $R=\OR\otimes_\Z S$ and note that since $\Z$ and $(\OR, \bar{\beta})$ are bi-interpretable, $S$ and $R$ are also bi-interpretable, and by Theorem~\ref{aut-biint:thm} there are only countably many automorphisms of $R$ fixing $\bar{\beta}$. Since $G'$ is uniformly definable in $G$, as the subgroup of products of commutators of width $w(n,\OR)$, and $G$ is bi-interpretable with $\OR$, all subgroups $T_{ij}$ are definable in $(\GLO,\bar{c})$. In particular $\OR$ is interpreted on $T_{ij}$ with the help of the constants $\bar{c}$. Moreover, for $R$ defined above, $R$ and $\GLR$ are bi-interpretable with the same constants as in the bi-interpretation of $\GLO$ and $\OR$. 
	
We note that for an element $\alpha$ of infinite order in the group of units $\OM$, which is finitely generated by Dirichlet's Units Theorem as an abelian group, the cyclic subgroup $\alpha^\Z\define\{\alpha^b:b\in \Z\}$ is definable in $\OR$ with some parameters. This is because for a $\Z$-basis of $\OR$ the integer exponentiation of $\alpha$ is computed by recursive functions in each coordinate, therefore it is arithmetic in each coordinate, since $\OR$ with these parameters is bi-interpretable with $\Z$. This also endows $\alpha^\Z$ with a ring structure isomorphic to $\Z$. 

Now, there is a unique element $y$ of $T_{ij}$ such that $d_i(\alpha)t_{ij}d_i(\alpha)^{-1}=y$, and we denote it for obvious reasons by $t_{ij}(\alpha)$. Therefore, the set $t_{ij}(\alpha^\Z)$ along with its ring structure are definable in $G$. Hence, $H_i=d_i(\alpha^\Z)\cdot Z(G)$ is definable in $G$ as follows: $h\in H_i$ if and only if $h^{-1} t_{ij} h \in t_{ij}(\alpha^\Z)$, for all $i\neq j$, and $h^{-1} t_{kj} h=1$ if $k\neq i$ and $j\neq i$. So $H_i$ is definable in $G$ via bi-interpretability of $G$ and $\Z$. Moreover, $H_i/Z(G)$ inherits an arithmetic structure, which is interpretable in $G$. 

Now, working in $G^*=\GLR$, the same formulas as above define subgroup $H_i^*$ of $G^*$, where $H^*/Z(G^*)\cong \langle S,+,\cdot \rangle$. Since $S$ is a countable non-standard model of $\Z$ its additive groups splits as $A \oplus D$, where $D\cong \Q^\omega$, $\Q$, the additive group of rational fractions. The group $\Q^\omega$ is a divisible abelian group and splits from the rest of the abelian subgroup $d_i(R^\times)$. By a similar argument $Z(G^*)=\{\prod_{i=1}^nd_i(\alpha): \alpha \in R^\times\}$ contains subgroups $A'$ and  $D'\cong \Q^\omega$, where $Z(G^*)\cong A'\oplus D'$. Since there are uncountably many distinct homomorphisms from $\Q^\omega$ onto $\Q^\omega$, there exist uncountably many distinct non-trivial homomorphisms $\phi_i: d_1(R^\times) \to Z(G^*)$ fixing the standard copy of $d_1(\OM)$ in $d_1(R^\times)$. For each such $i$ define a map:
  $$\psi_i:\SLR \rtimes d_1(\RM) \to  \SLR \rtimes d_1(\RM), \quad (x,y) \mapsto (x, y \phi_i(y)).$$
  $\psi_i's$ are pairwise distinct and each is a non-trivial automorphism of $\GLR\cong  \SLR \rtimes d_1(\RM)$, fixing $\GLO$ elementwise, inducing identity on all $T_{ij}$ and hence on $(R,\bar{\beta})$. Recall that by our hypothesis and Theorem~\ref{aut-biint:thm}, there has to exist a one-one correspondence between automorphisms of $(\GLR, \bar{c}, \bar{e})$ and automorphisms of $(R,\bar{\beta})$, where the former is uncountable, but the latter is countable. Contradiction!  \end{proof} 
\subsection{ Models of the complete theory of $\GLO$}
Here we prove that all models of the first-order theory of $\GLO$ are of the type $\text{GL}_n$. We invite the reader to compare the result with the case of $\TO$.

\begin{thm}If $H$ is any group such that $H\equiv \GLO$ then $H\cong \GLR$ for some ring $R\equiv \OR$. \end{thm}

\begin{proof} Let $G=\GLO$. Since $G'=\SLO$ and $\SLO$ is boundedly generated, $G'$ is uniformly definable in $G$. By Theorem~\ref{sln-elemchar:thm} if $H\equiv G$, then $H'\cong \SLR$ for a ring $R\equiv \OR$. Subgroups $\Delta_i=d_i(\OM)\cdot Z(G)$ are definable in $G$ similar to the proof of Theorem~\ref{TnMain:thm}. Indeed the same formulas define $d_i(\RM)\cdot Z(H)$ in $H$ where $Z(H)\cong \RM$. The action of $\Delta_i(R)$ on $\SLR$ is also an extension of the action of $\Delta_i(\OR)$ on $\SLO$. Now to prove the theorem we need to prove that torus $E_n=\Delta_1\cdots \Delta_n$ is actually a split torus, i.e. $E_n\cong d_1(\RM)\times \cdots \times d_n(\RM)$, or equivalently that $f_i\equiv 1$ for all $i$. By Remark~\ref{cocycles:rem}, we note that CoT 2-cocycles $f_i$ defining $\Delta_i(R)$ satisfy 
	\begin{equation}\label{cocyc-prod:eqn} f_n\equiv f_1^{-1}\cdots f_{n-1}^{-1}\end{equation}
	
	 On the other hand the following relations hold in $G$:
	 
	 $$d_i(\beta)d_j(\beta^{-1})=t_{ij}(\beta)t_{ji}(-\beta^{-1})t_{ij}(\beta)t_{ij}(-1)t_{ji}(1)t_{ij}(-1)), \quad i\neq j$$
	 
	Hence, there exists an $\Lgroups$-sentence which holds in $H$ and implies in $H$ that the set $\Delta_{ij}$ of elements $d_i(\beta)d_j(\beta^{-1})$ as $\beta$ ranges over $\RM$ is indeed a subgroup of $E_n$ which intersects $Z(H)$ trivially. A couple of routine calculations with symmetric 2-cocycles reveals that $$f_i(\alpha,\beta)f_i(\alpha^{-1},\beta^{-1})\equiv f_i^{-1}(\alpha^{-1}\beta^{-1},\alpha\beta)f_i(\alpha^{-1},\alpha)f_i(\beta^{-1},\beta)\equiv \one$$
	This together  with the fact that $\Delta_{ij}$ splits over $Z(H)$ implies that \begin{equation} \label{cohom-ij:eqn} f_i\equiv f_j, \quad i\neq j\end{equation}
	Now, Equations~\eqref{cocyc-prod:eqn} and \eqref{cohom-ij:eqn} imply that for any $i$
		$$f_i^{n}\equiv \one$$ 
	implying that $f_i\equiv \one$, since $f_i$'s are CoT. \end{proof}

\end{document}